\documentclass[times]{elsarticle}
\usepackage{color}
\usepackage{amssymb}
\usepackage{geometry}
\usepackage{amsthm}
\usepackage{amsmath}
\usepackage{graphicx}
\usepackage{caption}
\usepackage{float}
\usepackage{subfigure}
\usepackage{diagbox}
\usepackage{stfloats}
\usepackage{enumerate}
\usepackage{sectsty}
\newtheorem{theorem}{Theorem}
\usepackage{epstopdf}
\newtheorem{remark}{Remark}
\newtheorem{lemma}{Lemma}
\usepackage{titlesec}
\usepackage{titletoc}
\sectionfont{\centering}

\usepackage{algorithm,algpseudocode,float}
\usepackage{lipsum}

\newtheorem{example}{Example}[section]

\begin{document}

\begin{frontmatter}
\title{Selecting the Optimal Parameters Results in Double Interpolation: Double AFD}

\author[author1]{Tao Qian}
	\ead{tqian@must.edu.mo}
\author[author1]{Yunni Wu}
	\ead{wyn0109@163.com}
\author[author3]{Wei Qu}
\ead{quwei2math@qq.com}
\author[author4]{Yanbo Wang\corref{cor1}}
\ead{ybwang@wtu.edu.cn}

\address[author1]{Macau Center for Mathematical Sciences, Macau University of Science and Technology, China}
\address[author3]{College of Sciences, China Jiliang University, China}
\address[author4]{Faculty of Mathematics and Statistics, WuHan Textile University, China.}
\cortext[cor1]{Corresponding author:Yanbo Wang}

\begin{abstract}
Let $f$ belong to the Hardy space $H^2(\mathbb{D})$ of the unit disc, and $e_a$ the normalized Szeg\"o (reproducing) kernel of $H^2(\mathbb{D}).$ It is well known that, due to the reproducing kernel property, for any distinct $n$ points $a_1,\cdots,a_n$ in $\mathbb{D}$ the orthogonal projection of $f$ into ${\rm span}\{e_{a_1},\cdots,e_{a_n}\},$ denoted as $P_{{\rm span}\{e_{a_1},\cdots,e_{a_n}\}}(f),$ interpolates $f$ at the points $a_k$'s. The present study further proves that if the $a_k$'s are optimally selected according to certain energy matching pursuit principle, then $P_{{\rm span}\{e_{a_1},\cdots,e_{a_n}\}}(f)$ double interpolates $f$ at the points $a_k$'s, or order $m=2$ interpolation, that is,
\[     P_{{\rm span}\{e_{a_1},\cdots,e_{a_n}\}}(f)(a_k)=f(a_k), \quad {\rm and}\quad
P_{{\rm span}\{e_{a_1},\cdots,e_{a_n}\}}'(f)(a_k)=f'(a_k),\quad k=1,\cdots,n.\] With the accordingly newly defined double Takenaka-Malmquist system, the norm convergence for $n\to \infty,$ the $n$-best approximation for $n$ being fixed, and the related boundary function interpolation are studied.  The such generated new sparse representation, named as double AFD, is shown to outperform the classical AFD. Pointwise interpolations for orders $m>2,$ meaning to simultaneously interpolates all functions $f,f',\cdots,f^{(m-1)}$ at a set of $a_k$'s are, additionally, discussed. For the Hardy space of the upper-half complex plane there exists a counterpart theory.

\end{abstract}

\end{frontmatter}

\section{Introduction: The Hardy space and AFD}
\setcounter{page}{1}
\pagenumbering{arabic}
The original AFD
 (\cite{QWa1}) is also called \emph{Core AFD}, or just AFD in the present paper.
AFD is formulated with the Hardy space of analytic functions of finite energy in the unit disc

$$H^2(\mathbb{D})=\lbrace f:\mathbb{D}\rightarrow \mathbb{C}:f=\sum^\infty_{k=0}c_kz^k, {\rm where}\ \
\|f\|^2=\sum^\infty_{k=0}\vert c_k\vert^2<\infty \rbrace,$$

where $\mathbb{D}$ stands for the open unit disc on the complex plane. It is a fundamental result that any
$f\in H^2(\mathbb{D})$ has non-tangential boundary limits a.e. on the unit circle.
There will be no ambiguity if the boundary limit function is still denoted $f.$
On the boundary the inner product of $f,g\in L^2(\partial \mathbb{D})$ is defined to be
\begin{align}\label{a}
\langle f,g\rangle=\frac{1}{2\pi}\int^{2\pi}_0 f(e^{it})\overline{g}(e^{it})dt.
\end{align}
The mapping that maps
$f\in H^2(\mathbb{D})$ to its boundary limit function is an isometric isomorphism between
  $H^2(\mathbb{D})$
and a closed subspace of
$L^2(\partial \mathbb{D}),$ denoted as $H^2(\partial \mathbb{D}).$ Importance of the
Hardy spaces in functional analysis lays in the fact that there exist direct sum decompositions of
the boundary spaces into the Hardy spaces on the two sides of the boundary. As an example,
 in the boundary value sense,
$$ L^2(\partial \mathbb{D})\cong H^2(\partial \mathbb{D})\oplus H^2(\partial\overline{\mathbb{D}}^c)$$
corresponding to the relations
$$f=f^++f^-, \quad f^+\in H^2(\partial\mathbb{D}), \quad f^-\in H^2(\partial\overline{\mathbb{D}}^c).$$
Moreover, if $f\in L^2(\partial \mathbb{D})$ is real-valued, then
\[ f=2{\rm Re}f^+-c_0, \quad f^+={\rm Proj}_{H^2(\mathbb{D})}f,
\quad c_0\ {\rm is \ the\  0-the \ Fourier\  coefficient\ of}\ f.\]
  For related details see \cite{QWa1}.
We note that,  $L^2(\partial \mathbb{D})$ and $H^2(\partial \mathbb{D})$ are not, but
 $H^2(\mathbb{D})$ is a reproducing kernel Hilbert space (RKHS), where the reproducing kernel is
$$ k_a(z)=\frac{1}{1-\overline{a}z}, \qquad a\in \mathbb{D}.$$
$k_a(z)$ is also called the $a$-parameterized Szeg\"o kernel.  There holds the reproducing kernel property:
for $f\in H^2(\mathbb{D}),$
$$          \langle f,k_a\rangle =f(a).$$

Below, we will be only working with the Hardy and the boundary Hardy spaces, and,
to save the notation, we write $f$ for, identically, the Hardy space projection function $f^+.$ Given a sequence
$\lbrace a_k\rbrace^\infty_{k=1}\subset \mathbb{D}$ with multiplicity allowed, we have a
sequence of \emph{multiple Szeg\"o kernels}
 $\lbrace \tilde{k}_{a_k}\rbrace^\infty_{k=1},$
 where $$\tilde{k}_{a_k}=\left[\left(\frac{\partial k_a}{\partial \overline{a}}\right)^{l(k)-1} \right]_{a=a_k},$$
 where $l(1)=1,$ and $l(k)$ is the multiple of $a_k$ in $(a_1,\cdots,a_k).$

The Gram-Schmidt orthogonalization of sequence $\{\tilde{k}_{a_k}\}$ forms the corresponding  Takenaka-Malmquist (or T-M) system $\lbrace B_k=B_{a_1,\cdots,a_k}\rbrace^\infty_{k=1},$ where
$$B_k(z)=\frac{\sqrt{1-\vert a_k\vert^2}}{1-\bar{a}_kz}\prod^{k-1}_{l=1}\frac{z-a_l}{1-\bar{a}_lz}$$
(see Appendix of \cite{QSAFD}).

The $H^2$-normalization of $k_a$ is
\begin{align}\label{b}
e_a=\frac{\sqrt{1-\vert a\vert^2}}{1-\bar{a}z},
\end{align}

and thus

\begin{align}\label{c}
\langle f,e_a\rangle=\sqrt{1-\vert a\vert^2}f(a).
\end{align}

 {AFD}, or {Core AFD},  is to adaptively select optimal parameters $a_1,\cdots,a_n,\cdots,$ such that
\begin{align}\label{core}
\vert\langle f_k, e_{a_k}\rangle\vert=\max\lbrace \vert\langle f_k, e_a\rangle\vert, a\in\mathbb{D}\rbrace,\qquad k=1,\cdots,n,\cdots
\end{align}
 where  $f_k,$ the so called
\emph{single reduced remainders}, or reduced remainders as used in the previous literature, are defined to be
\begin{align} \label{reduced}
 \ f_k(z)\triangleq \frac{f_{k-1}(z)-\langle f_{k-1},e_{a_{k-1}}\rangle e_{a_{k-1}}(z)}{
\frac{z-a_{k-1}}{1-\overline{a}_{k-1}z}},\quad k=2,3,\cdots \end{align}
with $f_1=f.$
Since, due to reproducing kernel property, $f_{k-1}(z)-\langle f_{k-1},e_{a_{k-1}}\rangle e_{a_{k-1}}(z),$  the numerator part of (\ref{reduced}), has zero $a_{k-1},$  by dividing the corresponding M\"obius transform, we have that  $f_k$ still belongs to $H^2(\mathbb{D}).$ Attainability of the optimal parameter $a_k$ inside the disc is proved in  \cite{QWa1}.
By using the optimally selected parameters $a_k$'s we obtain a fast convergent series expansion of $f.$
 \begin{equation}\label{d}
    \begin{aligned}
                                f=\sum^\infty_{k=1}\langle f_k,e_{a_k}\rangle B_k,
    \end{aligned}   \end{equation}
    where
    \begin{align}\label{TM} B_k(z)=e_{a_k}(z)\prod_{l=1}^{k-1}\frac{z-a_{l-1}}{1-\overline{a}_{l-1}z},
     k=1,2,\cdots  \end{align}
      is the associated Takenaka-Malmquist (TM) orthonormal system.
 Availability of optimal selections of the parameters, fast
convergence, as well as positivity of frequency of each term of the TM system
are proved in \cite{QWa1}.

Due to the orthogonality property of the system $\{B_k\}$ there hold

\begin{align}\label{e}
\Vert f\Vert^2
=\sum^\infty_{k=1}\vert\langle f,B_k\rangle\vert^2=\sum^\infty_{k=1}\vert\langle f_k,e_{a_k}\rangle\vert^2
=\sum^\infty_{k=1}(1-\vert a_k\vert^2)\vert f_k(a_k)\vert^2.
\end{align}

We in particular note the useful relations
$$\langle f, B_k\rangle=\langle g_k, B_k\rangle=\langle f_k, e_{a_k}\rangle,$$
where
\begin{eqnarray}\label{g}
g_k=f-\sum^{k-1}_{l=1}\langle f_l,e_{a_l}\rangle B_l\end{eqnarray} is the $k$'s \emph{standard (orthogonal) remainder} of $f$ with respect to the TM system $\{B_l\}.$ In other words, $g_k$ is the projection of $f$ into the orthogonal complement of ${\rm span}\{B_1,\cdots, B_{k-1}\}.$

If a function $f$ can be explicitly
expressed as
$$f=\sum^n_{k=1}\langle f_k, e_{a_k}\rangle B_k,$$
then we say that $f$ is an \emph{$n$-Blaschke form}.

\begin{remark}
Notice that, due to the reproducing kernel property, for any $a_1,\cdots, a_n$ in $\mathbb{D}$ not necessarily optimally selected, we have the identical relation
\begin{align}\label{e}
f=\sum^n_{k=1}\langle f_k, e_{a_k}\rangle B_k+f_{n+1}(z)\prod_{k=1}^n\left(\frac{z-a_k}{1-\bar{a}_kz}\right),\end{align}
where $f_{n+1}$ is a function in $H^2(\mathbb{D}).$

We thus have the interpolation property,
for $g_{n+1}=f-\sum^n_{k=1}\langle f_k, e_{a_k}\rangle B_k$, there holds
\begin{align}\label{core1}
g_{n+1}(a_k)=0,\qquad k=1,\cdots,n.
\end{align}

At such situation, however, one cannot ensure the convergence
\begin{eqnarray}\label{above} \|g_{n+1}\|=\|f-\sum_{k=1}^n \langle f,B_k\rangle B_k\|\to 0,\end{eqnarray}
unless extra conditions being assumed upon $a_1,\cdots, a_n, \cdots.$

Under the optimal selections (\ref{core}) of the $a_k$'s the above convergence (\ref{above}) holds (\cite{QWa1}). Furthermore, for $f$ in the space $H^2_M(\mathbb{D}), M>0,$\ defined in (\ref{M}) below , there holds
 \begin{align} \label{M} \|g_{n+1}\|\leq \frac{M}{\sqrt{n}}.\end{align}
 The above quantitative estimate is established in \cite{QWa2} following the seminal techniques established in \cite{DT}.
   \end{remark}

As a new discovery, the present paper shows
 \begin{align}\label{e2}
 f=\sum^n_{k=1}\langle \tilde{f}_k, e_{a_k}\rangle
 \tilde{B}_k+\tilde{f}_{n+1}(z)\prod_{k=1}^n\left(\frac{z-a_k}{1-\bar{a}_kz}\right)^2,\end{align}
 where $a_k$ are optimally selected,
 \begin{eqnarray}\label{td} a_k=\arg \max \lbrace \vert\langle \tilde{f}_k, e_a\rangle\vert, a\in\mathbb{D}\rbrace, \end{eqnarray}  and each $\tilde{B}_k$ contains correspondingly double zeros forming what we call \emph{double-zero- or D-TM system}:
 \begin{eqnarray}\label{dTM} \{\tilde{B}_{k}(z)\}_{k=1}^\infty \triangleq \{e_{a_k}(z)\prod_{l=1}^{k-1}\left(\frac{z-a_{l-1}}{1-\overline{a}_{l-1}z}\right)^2\}_{k=1}^\infty,\qquad z\in {\mathbb D},\end{eqnarray}
 and  $\tilde{f}_k$ are called
\emph{double reduced remainders}, defined
\begin{align} \label{rreduced}
 \ \tilde{f}_k(z)\triangleq \frac{\tilde{f}_{k-1}(z)-\langle \tilde{f}_{k-1},e_{a_{k-1}}\rangle e_{a_{k-1}}(z)}{
\left(\frac{z-a_{k-1}}{1-\overline{a}_{k-1}z}\right)^2},\quad k=2,3,\cdots \end{align}
with $\tilde{f}_1=f_1=f.$ The validity of the definition $\tilde{f}_k$ does not only rest on the reproducing kernel property, but also on the optimal selections of $a_{k-1}.$ As a result, unlike the identity (\ref{e}) for all cases,  the new relation (\ref{e2}) is not an identity but strictly depends on optimally selected $\{a_k\}_{k=1}^n\in {\mathbb D}^n.$  As a reward, however, the relation (\ref{e2}) amounts to the double interpolation property at the selected $a_k$'s which strengthens the existing AFD approximation.  \\

We show there further holds
\begin{align}\label{e3}
 f=\sum^\infty_{k=1}\langle \tilde{f}_k, e_{a_k}\rangle
 \tilde{B}_k,\end{align}
 with the same convergent rate if $f\in H^2_M(\mathbb{D}).$

\begin{remark}
We note the following observations: $1^\circ.$ As is shown below, being more sparse than that of AFD, (\ref{e2}) gives rise to a fast convergent reconstruction of the originally given signal $f;$  and $2^\circ.$ If selecting $a_1=0, a_2$ is optimal according to (\ref{core}), and the rest $a_k, k=3,\cdots,$ are according to (\ref{td}), then (\ref{e2}) is an intrinsic mono-component decomposition, as a linear combination of $H^2$ functions having at a.e. points on the boundary non-negative phase derivatives (\cite{QianPhase,Qinner}) (also called non-negative analytic instantaneous frequencies). It can be explicitly written out, for $f_1=f, f_2(z)=\frac{f_1(z)-f_1(0)}{z},
\tilde{f}_3(z)=\frac{f_2(z)-\langle f_2,e_{a_2}\rangle e_{a_2}(z)}{\frac{z-a_2}{1-\overline{a}_2z}},$
\begin{eqnarray}\label{intrin}  f(z)=f(0)+\langle f_2,e_{a_2}\rangle  ze_{a_2}(z)+\sum_{k=3}^{\infty}\langle \tilde{f}_k,e_{a_k}\rangle ze_{a_k}(z)\prod_{l=2}^{k-1}\left(\frac{z-a_l}{1-\overline{a}_lz}\right)^2,\end{eqnarray}
\end{remark}
 where $\tilde{f}_k, k>3,$ are according to (\ref{rreduced}). To see each term is a mono-component we first note that functions of the form $ze_{a}(z)$ are starlike functions with pole $0,$ and hence have non-decreasing continuous phases and therefore have non-negative phase derivatives. Secondly, phase derivatives of boundary values of M\"obius transforms are Poisson kernels which are pointwise positive. Finally, the set of mono-components is closed under the multiplication operation. Due to the sparsity we say that (\ref{intrin}) is an intrinsic mono-components decomposition of $f.$ \\

 The whole paper writing plan is as follows. In \S 2 we validate the algorithm formulation of Double Zero Adaptive Fourier Decomposition, abbreviated as Double-AFD or D-AFD. To this end we prove a main technical lemma. We include a second lemma spelling out the relations between different types of remainders. They will be used to simplify computation and to prove the $n$-Best D-AFD type expansion in the following sections. \S 3  will prove convergence and \S 4 will deduce a convergence rate of D-AFD. In \S 5 we will prove the induced boundary interpolation properties. In \S 6 we will develop the $n$-Best D-AFD approximation. \S 7 addresses superperformance of  D-AFD in reconstructions. We include a number of examples. \S 8 extends partial results of D-AFD, corresponding to $m=2,$ to the $m>2$ cases, and raises some questions for future studies.

\section{Formulation of Double Zero Adaptive Fourier Decomposition, or Double-AFD}

The relation (\ref{e}) holds, in fact, for arbitrary $a_1,\cdots,a_n$ in $\mathbb{D},$
with an application of the reproducing kernel property of $k_a.$ The
optimality of selections of the $a_k$'s is used only when proving
the convergence itself, as well as when deducing the convergence rate (\ref{d}) and (\ref{M}).

We will be working with, in place of TM systems (\ref{TM}),
\emph{double zero-TM systems}, or\ $D$-TM systems, $\{\tilde{B}_k\}_{k=1}^\infty,$ in brief,  where
$\tilde{B}_1(z)=e_{a_1}(z),$ and
\begin{align}\label{dTM}  \tilde{B}_k(z)=e_{a_k}(z)
\prod_{l=1}^{k-1}\left(\frac{z-a_l}{1-\bar{a}_lz}\right)^2, \quad k=2,\cdots.\end{align}
 For any positive integer $N, \{\tilde{B}_k\}_{k=1}^N$  is itself a
 particular finite TM system, called $N$-\emph{double TM system}, and hence an orthonormal system. The term
 $\tilde{B}_k$ is also denoted as $\tilde{B}_k=\tilde{B}_k^{a_k}$ to specially indicate dependence on the last parameter. This notation brings some convenience in further discussions of D-AFD.

We need the following technical lemma.

\begin{lemma} \label{ccore}
  Let $f\in H^2(\mathbb{D})$ be non-trivial and $a_1\in \mathbb{D}$ such that
  \begin{align}\label{optima}
|\langle f,e_{a_1}\rangle|=\max \{|\langle f,e_{a}\rangle| \ :\ a\in \mathbb{D}\}.\end{align}
  Then the single reduced remainder $f_2(z)$ in accordance of (\ref{reduced}) again has $a_1$ as its zero. Or, precisely, there exists $\tilde{f}_2\in  H^2(\mathbb{D})$ such that $$f_2(z)=\frac{f_1(z)-\langle f_1,e_{a_1}\rangle e_{a_1}(z)}{\frac{z-a_1}{1\overline{a}_1z}}=\frac{z-a_1}{1-\overline{a}_1z}\tilde{f}_2(z),$$ where
 $f_1=f.$ As a consequence, for the standard remainder $g_2$ defined in (\ref{g}), there holds
 $$g_2(z)=\left(\frac{z-a_1}{1-\overline{a}_1z}\right)^2\tilde{f}_2(z).$$
  \end{lemma}

\textit{Proof of Lemma \ref{ccore}.}
Since $f_2\in H^2(\mathbb{D}),$ we only need to show
$\lim_{z\to a_1}f_2(z)=0.$
For $z\ne a_1,$
 \begin{eqnarray}   \label{zer}
                  f_2(z)&=&\frac{f(z)-\langle f,e_{a_{1}}\rangle e_{a_{1}}(z)}
       {\frac{z-a_{1}}{1-\overline{a}_{1}z}}\nonumber \\
       &=&\frac{(1-\overline{a}_{1}z)f(z)-(1-|a_1|^2)f(a_1)}{z-a_1}\nonumber\\
       &=&-\overline{a}_1f(z)+
       (1-|a_1|^2)\frac{f(z)-f(a_1)}{z-a_1}\nonumber\\
       &\to & -\overline{a}_1f(a_1)+(1-|a_1|^2)f'(a_1),\qquad {\rm as} \ z\to a_1.
       \end{eqnarray}
       Now we show that the last quantity (\ref{zer}) is zero.

Due to the optimal selection of $a_1$ in accordance with (\ref{optima}) there hold
\begin{align}
0&=\frac{\partial}{\partial z} \left[|\langle f,e_z\rangle|^2\right]_{z=a_{1}}\notag\\
&=\frac{\partial}{\partial z} \left[(1-\vert z\vert^2)\vert f(z)\vert^2\right]_{z=a_{1}}\notag\\
&=\left[\left(-\bar{z}f(z)+(1-\vert z\vert^2)f{'}(z)\right)\overline{f(z)}\right]_{z=a_1}\notag\\
&=\left[-\bar{a}_1f(a_1)+(1-\vert
a_1\vert^2)f{'}(a_1)\right]\overline{f(a_1)}.\notag\end{align}
Since $f(a_1)\ne 0,$ we conclude that the quantity (\ref{zer}) being zero is consequence of the optimality.
The proof is complete.  \\

According to Lemma \ref{ccore}, with $\tilde{f}_1=f$ and an optimally selected
$a_1\in \mathbb{D}$ according to (\ref{core}),
there holds
 \[ f(z)=\langle \tilde{f}_1,e_{a_1}\rangle e_{a_1}(z)+
 \tilde{f}_2(z)\left(\frac{z-a_1}{1-\overline{a}_1z}\right)^2.\]
Then to $\tilde{f}_2$ applying Lemma \ref{ccore} again, we have
$$ \tilde{f}_2(z)=\langle \tilde{f}_2,e_{a_2}\rangle e_{a_2}(z)+
\tilde{f}_3(z)\left(\frac{z-a_2}{1-\overline{a}_2z}\right)^2.$$
Subsequently,
 $$ f(z)=\langle \tilde{f}_1,e_{a_1}\rangle e_{a_1}(z)+
 \langle \tilde{f}_2(z),e_{a_2}\rangle e_{a_2}(z)\left(\frac{z-a_1}{1-\overline{a}_1z}\right)^2+
 \tilde{f}_3(z)\left(\frac{z-a_1}{1-\overline{a}_1z}\right)^2\left(\frac{z-a_2}{1-\overline{a}_2z}\right)^2,$$
  where $a_2$ is selected according to (\ref{core}) with $f$ being replaced by $\tilde{f}_2,$ and
  $$\tilde{f}_3(z)=\frac{\tilde{f}_2(z)-\langle \tilde{f}_2(z),e_{a_2}\rangle e_{a_2}(z)}
 {\left(\frac{z-a_2}{1-\overline{a}_2z}\right)^2}.$$
  Likewise, through iteration we have

\begin{align}\label{second}
f(z)=\sum^n_{k=1}\langle \tilde{f}_k,e_{a_k}\rangle \tilde{B}_k(z)+\tilde{f}_{n+1}(z) \prod_{k=1}^{n}\left(\frac{z-a_k}{1-\overline{a}_k z}\right)^2,
\end{align}
   where
   \begin{align}\label{optimal}
 a_k=\arg \max \{|\langle \tilde{f}_k,e_{a}\rangle| \ :\ a\in \mathbb{D}\}\end{align}
 are attainable, and
 \begin{align}\label{dre} \tilde{f}_{1}(z)=f(z), \ \tilde{f}_{k+1}(z)=
 \frac{\tilde{f}_k(z)-\langle \tilde{f}_k,e_{a_k}\rangle e_{a_k}(z)}
 {\left(\frac{z-a_k}{1-\overline{a}_k z}\right)^2}, \quad k=1, 2,\cdots \end{align}
 are \emph{double reduced reminders}.

   The formed series decomposition is named as
   \emph{Double Zero AFD}, or \emph{Double AFD}, or \emph{D-AFD } in brief.
 In the following sections we study convergence properties of
 \begin{align}\label{second}
 f(z)=\sum^\infty_{k=1}\langle \tilde{f}_k,e_{a_k}\rangle \tilde{B}_k(z),
  \end{align}
  where the orthonormal system $\{\tilde{B}_k\}$ is defined by (\ref{dTM}).

  \begin{remark}\label{ak}
  For a given $f\in H^2(\mathbb{D}),$ to perform Core AFD we produce a sequence of parameters $a_k$ in the unit disc. If performing Double AFD we produce, essentially, a different sequence of parameters $a_k$ in the disc. The former is according to (\ref{core}) with respect to the (single) reduced remainder $f_k$ given by (\ref{reduced}); while the latter is according to (\ref{optimal}) with respect to the double reduced remainder $\tilde{f}_k$ given by (\ref{rreduced}). At each use of $a_k$ from the context there should be no ambiguity.
  \end{remark}

\begin{remark}
The proof of Lemma \ref{ccore} is based on $\frac{\partial}{\partial z}|\langle f,e_{z}\rangle|^2=0$ at $z=a_1,$ that is only a necessary condition of the maximal selection principle (\ref{optima}). In the proof of convergence and convergence rate of Double AFD the full strength of (\ref{optima}) will be used.

The proof  of Lemma \ref{ccore} implies that the negation proposition is also true that amounts saying that if the used parameters do not meet the optimal selection principle (\ref{optima}), then there does not hold a  Double-AFD expansion.
\end{remark}

\begin{remark}
The strength  of double interpolation is fully contained in formula (\ref{second}): Since there holds
\[\tilde{g}_n(z)=f(z)-\sum^{n-1}_{k=1}\langle \tilde{f}_k,e_{a_k}\rangle \tilde{B}_k(z)=\tilde{f}_{n}(z) \prod_{k=1}^{n-1}\left(\frac{z-a_k}{1-\overline{a}_k z}\right)^2,\]
at all the optimally selected distinct parameter points $a_1,\cdots,a_{n-1},$ the projection of $f$ into ${\rm span}\{\tilde{B}_k\}_{k=1}^{n-1},$ as a rational function of degree $2n+1$ with an order one pole at the infinity, second-order-interpolates $f$ at the selected points. Precisely,
\[ f(a_k)=\tilde{g}_n(a_k), \quad f'(a_k)=\tilde{g}_n'(a_k),\quad k=1,\cdots,n-1.\]
\end{remark}

For further development we will prove the following

\begin{lemma} \label{3=}
\begin{align}\label{must} \langle \tilde{f}_k,e_{a_k}\rangle=\langle \tilde{g}_k,\tilde{B}_k\rangle=\langle f,\tilde{B}_k\rangle,\end{align}
where, with $\tilde{g}_1=f,$
 $$\tilde{g}_k=f-\sum_{j=1}^{k-1}\langle f,\tilde{B}_j\rangle \tilde{B}_j, \quad k\ge 2,$$ is the projection of $f$ into the span of $\tilde{B}_1,\cdots,\tilde{B}_{k-1}$ determined by the consecutively selected optimal $a_1,\cdots,a_{k-1}$ in $\mathbb{D}.$
\end{lemma}
\textit{Proof of Lemma \ref{3=}.} Due to the orthogonality between the $\tilde{B}_j, j=1,\cdots,k,$ we have
$\langle g_k,\tilde{B}_k\rangle=\langle f,\tilde{B}_k\rangle.$ To show the left-hand-side equal relation of
(\ref{must}) we use induction. With the convention  $f=f_1=\tilde{f}_1=\tilde{g}_1=g_1,$ we have
\begin{eqnarray*}
\frac{\tilde{g}_2}{\phi^2_{a_1}}=\frac{f-\langle f,e_{a_1} \rangle e_{a_1}}{\phi^2_{a_1}}=\tilde{f}_2,\quad {\rm with}\quad \phi_{a_1}(z)=\frac{z-a_1}{1-\overline{a}_1z}.
\end{eqnarray*}
This gives
\[ \langle \tilde{f}_2,e_{a_2}\rangle=\langle \tilde{g}_2,\tilde{B}_2\rangle.\]
Next we assume
\begin{eqnarray}\label{fract} \frac{\tilde{g}_{k}}{\phi^2_{a_1}\cdots\phi^2_{a_{k-1}}}=\tilde{f}_{k}, \end{eqnarray}
which implies \[ \langle
\tilde{f}_{k},e_{a_{k}}\rangle=\langle \tilde{g}_{k},\tilde{B}_{k}\rangle.\]
Consequently,
\begin{eqnarray*}
\tilde{f}_{k+1}=\frac{\tilde{f}_k-\langle \tilde{f}_k,e_{a_k} \rangle e_{a_k}}{\phi^2_{a_k}}=\frac{\tilde{g}_k-\langle \tilde{g}_k,\tilde{B}_k\rangle \tilde{B}_k}{\phi^2_{a_1}\cdots\phi^2_{a_k}}=\frac{\tilde{g}_{k+1}}{\phi^2_{a_1}\cdots\phi^2_{a_k}},
\end{eqnarray*}
implying
\[ \langle \tilde{f}_{k+1},e_{a_{k+1}}\rangle=\langle \tilde{g}_{k+1},\tilde{B}_{k+1}\rangle.\]
The proof is complete.

\section{Convergence of D-AFD}

\begin{theorem}\label{ConD}
(Convergence of D-AFD) Let $f\in H^2(\mathbb{D})$ and a sequence of double reduced remainders
$\tilde{f}_k$ are defined in (\ref{dre}) through the optimally selected parameters $a_k$'s in accordance of the criterion (\ref{optimal}). Then there holds, in the $H^2(\mathbb{D})$ sense,

\begin{align}\label{series}
f=\sum^\infty_{k=1}\left\langle \tilde{f}_k, e_{a_k}\right\rangle \tilde{B}_k,\end{align}

where $\{\tilde{B}\}_{k=1}^\infty$ is the double TM system (\ref{dTM}) generated by the selected $a_k$'s.
\end{theorem}
A proof can be drawn as, essentially, the same as that for AFD (see \cite{QWa1}). Below we provide a more intuitive and shorter proof.

\textit{Proof of Theorem \ref{ConD}.}
Since $\{\tilde{B}_k\}_{k=1}^\infty$ is an orthonormal system, by the Riesz-Fisher Theorem, the right-hand-side of (\ref{series}) converges to a function $\tilde{h}\in H^2(\mathbb{D}).$ We hence have
\[ f=\tilde{h}+\tilde{g}, \ \tilde{h}=\lim_{n\to \infty}\tilde{h}_n, \ \tilde{g}=\lim_{n\to \infty}\tilde{g}_n,\]
where
\[ \tilde{h}_n=\sum^{n-1}_{k=1}\left\langle f, \tilde{B}^{a_k}_k\right\rangle \tilde{B}^{a_k}_k,\quad
\tilde{g}_n=f-\sum^{n-1}_{k=1}\left\langle f, \tilde{B}^{a_k}_k\right\rangle \tilde{B}_k^{a_k}.\]
The notation $\tilde{B}^{a_k}_k$ is defined after introducing the double TM system in (\ref{dTM}).
Since
\[ \langle f-\tilde{h},\tilde{B}_k\rangle=0\]
for all $k,$ we have $\tilde{g}\in {\rm span}\{\tilde{B}_k\}^\perp$.  All needed to show is $\tilde{g}=f-\tilde{h}=0.$

Define $Q$ to be  the projection operator from $H^2(\mathbb{D})$ to $\{{\rm span}\{\tilde{B}_k\}_{k=1}^\infty\}^\perp,$ and $Q_n$ the projection
operator from $H^2(\mathbb{D})$ to
$\{{\rm span}\{\tilde{B}_k\}_{k=1}^n\}^\perp, n=1,2,\cdots$. Both of them are linear and self-adjoint.
In particular, since $\tilde{g}\in \{{\rm span}\{\tilde{B}_k\}_{k=1}^\infty\}^\perp\subset \{{\rm span}\{\tilde{B}_k\}_{k=1}^n\}^\perp,$  we have $Q_n\tilde{g}=\tilde{g}.$

 We show $\tilde{g}=0$ by contradiction. Assume  $\tilde{g}\ne 0.$ Then there would exist $a\in\mathbb{D}$ such that $\delta=|\langle \tilde{g},e_a\rangle|>0.$ Such $e_a$ must not be in ${\rm span}\{\tilde{B}_k\}_{k=1}^\infty.$
  As a consequence, $a$ does not coincide with any of the selected ${a_1},\cdots,a_n,\cdots $
We have the relations $\tilde{g}=Q_n\tilde{g}=Q_n(f-\tilde{h})$ and $Q_n\tilde{h}\to0,$ as $n\to \infty.$ There exists $N$ large enough such that
\begin{eqnarray*}
\delta &=&|\langle \tilde{g},e_a\rangle|\\&\leq&
|\langle Q_Nf,e_a\rangle|+|\langle Q_N(\tilde{h}),e_a\rangle|\\
&\leq&|\langle f,Q_N e_a\rangle|+\delta/4\\
&=&\sqrt{1-\sum^N_{k=1}|\langle e_a,\tilde{B}_k\rangle|^2}|\langle f,\tilde{B}_{N+1}^a\rangle|+\delta/4\\
&<&|\langle f,\tilde{B}_{N+1}^a\rangle|+\delta/4.\end{eqnarray*}
Therefore,
\[ |\langle f,\tilde{B}_{N+1}^a\rangle|>\delta/2.\]
This, however, exposes the contradiction that, according to the optimal selection principle, for $N$ large enough, we should select $a$ instead of the $a_{N+1}$ for which $|\langle f,\tilde{B}_{N+1}^{a_{N+1}}\rangle|\leq\delta/2$ from $\lim_{k\to \infty}|\langle f,\tilde{B}_k\rangle|=0.$ The proof is complete.

\section{Convergence Rate of D-AFD}

An order $n$ D-AFD creates a rational function having $2n$ interpolation points (including multiples) to the function to be approximated. It is, therefore, a more efficient approximation methodology than Core AFD, the latter having $n$ interpolation points (see \cite{LQ}). From the energy delay point of view in digital signal processing theory, factorising zero factors in the form of Blaschke products will speed up the convergence (\cite{CS} and \cite{Qinner}). Such effectiveness is demonstrated by experiments in the last part of this paper. Unfortunately, In terms of the convergence rate,  we can only achieve the same degree as that for AFD.  The proof for D-AFD follows the same route as for AFD (\cite{DT},\cite{QWa2}). Since D-TM systems are of a different construction, the result is now re-formulated and a proof is given below.

Let \begin{eqnarray}\label{M}H^2({\mathbb D},M)=\{ f\in H^2({\mathbb D})\ :\ f=\sum_{k=1}^\infty c_ke_{a_k},\sum_{k=1}^\infty |c_k|\leq M\}, \ 0<M<\infty.\end{eqnarray}

\begin{theorem}\label{rate}
If $f\in H^2({\mathbb D},M)$ then the $n$-standard orthogonal remainder $\tilde{g}_n$ for the double TM system
$\{\tilde{B}_k\}_{k=1}^\infty$ satisfies
\[\|\tilde{g}_n\|\leq \frac{M}{\sqrt{n}}.\]
\end{theorem}

\textit{Proof of Theorem \ref{rate}.}
By definition of $\tilde{g}_m$ and the orthogonality, as well as the relation (\ref{3=}), there holds
\[\|\tilde{g}_{m+1}\|^2=\|\tilde{g}_m\|^2-|\langle \tilde{g}_m,\tilde{B}_m\rangle|^2.\]
Since $f\in H^2({\mathbb D},M),$ there exists a sequence $\{b_k\}\subset {\mathbb D}$ and a sequence of complex numbers $\{c_k\}, \sum_{k=1}^\infty |c_k|\leq M,$ such that
$f=\sum_{k=1}^\infty c_ke_{b_k}$ in the $H^2({\mathbb D})$ sense.
Since $\tilde{g}_m$ is a projection, there follow
\begin{eqnarray}\label{non}
\|\tilde{g}_m\|^2&=&|\langle \tilde{g}_m,f\rangle|\nonumber\\
&=&|\langle \tilde{g}_m,\sum_{k=1}^{\infty}c_ke_{b_k}\rangle|\nonumber\\
&\leq& M\sup_{b_k}|\langle \tilde{g}_m,e_{b_k}\rangle|\nonumber\\
&=& M \sup_{b_k} \sqrt{1-|b_k|^2}|\tilde{g}_m(b_k)|.
\end{eqnarray}
By the maximal selection principle (\ref{optimal}) and the relation (\ref{fract}), there hold
\begin{eqnarray*}
|\langle \tilde{g}_m,\tilde{B}_m\rangle|&=&\sup_{a\in \mathbb{D}}|\langle \tilde{g}_m,\tilde{B}^a_m\rangle|\\
&=&\sup_{a\in \mathbb{D}}|\langle \tilde{f}_m,e_a\rangle|\\
&=&\sup_{a\in \mathbb{D}}\sqrt{1-|a|^2}|\tilde{g}_m(a)||\prod_{k=1}^{m-1}\frac{1-\overline{a}_ka}{a-a_k}|^2\\
&\ge & \sup_{b_k}\sqrt{1-|b_k|^2}|\tilde{g}_m(b_k)||\prod_{k=1}^{m-1}\frac{1-\overline{a}_kb_k}{b_k-a_k}|^2\\
&\ge & \sup_{b_k}\sqrt{1-|b_k|^2}|\tilde{g}_m(b_k)|\\
&\ge & \frac{\|\tilde{g}_m\|^2}{M},
\end{eqnarray*}
where the last step uses the estimate (\ref{non}). Then we have
\[\|\tilde{g}_{m+1}\|^2\leq \|\tilde{g}_m\|^2\left(1-\frac{\|\tilde{g}_m\|^2}{M^2}\right).\]
By taking $A=M^2$ in the following lemma(\cite{QWa2},also\cite{DT}in general),
\begin{lemma}
If a sequence of non-negative numbers $\{d_m\}_{m=1}^\infty$ satisfies
\[ d_1\leq A, \ d_{m+1}\leq d_m\left(1-\frac{d_m}{A}\right),\]
then
\[ d_m\leq \frac{A}{m}.\]
\end{lemma}
 The proof is complete.

\begin{remark}
It would expect that the convergence rate reaches some degree like $O(\frac{1}{\sqrt{2n}})$ but with the greedy type proof this was unable to achieve. We note that although the winding number of the remainder reaches $2n-2,$ the number of optimally selected parameters is, as in AFD, still $n.$ For Blaschke unwinding expansion, that may be related to the present topic, we refer the reader to   (\cite{CS,Qinner,CP}). It is also noted that the boundary function of a Hardy space function may not be smooth, and the Fourier coefficients should not show good regularity.  Incidently, the obtained order of convergence is the same as Shannon's interpolation and the Karhunen-Lo\'eve expansion for random signals.\end{remark}

\section{Interpolation on the Boundary}

In the present section we study the boundary interpolation property of the D-AFD.
We assume that $f$ stands for a real-valued signal. Then there holds
\begin{align}\label{real}
f=2{\rm Re}\{f^+\}-c_0,\end{align}
where $f^+=(1/2)(f+iHf)$ is the projection of $f$ into $H^2(\mathbb{D}),$ $H$ is the Hilbert transform, and $c_k, k=0,\pm 1,\pm 2,\cdots$ denote the complex  Fourier coefficients of $f.$ There holds the following result.

\begin{theorem}\label{ZC}(D-AFD Zero Crossing Theorem)

Assume that the phase function $\Psi_n$ of the real part of the $n$-remainder derived from (\ref{real})
\[ R_n(t)\triangleq f(e^{it})-2{\rm Re}\{\sum_{l=1}^n\langle f,\tilde{B}_l\rangle\tilde{B}_l(e^{it})\}+c_0=2{\rm Re}\tilde{f}_{n+1}(e^{it})\prod_{l=1}^{n}\left(\frac{e^{it}-a_l}{1-\overline{a}_le^{it}}\right)^2\triangleq\rho_n (t)\cos \Psi_n(t),\] where $\rho_n(t)\ge 0,$  is an absolutely
continuous $2\pi$-periodic function. Then $\cos \Psi_n$
has at least $4n$ zeros including multiples. As a consequence, the $n-$th remainder derived from the $n$-D-AFD series interpolates the original real-valued function $f$  at least at $4n$ points, including multiples.
\end{theorem}

With the preparations in \cite{QianPhase} (also see \cite{Gar}), the proof is straightforward, as given below.

\textit{Proof of Theorem \ref{ZC}.} We note that
\[ \frac{e^{it}-a_l}{1-\overline{a}_le^{it}}=e^{i\theta_{a_l}(t)},\]
where
\[ \theta'_{a_l}(t)=\frac{1-|a_l|^2}{|e^{it}-a_l|^2}>0\]
is the Poisson kernel with
\[ \int_0^{2\pi}\theta'_{a_l}(t)dt=2\pi, \quad l=1,2,\cdots\]
On the other hand, $\Psi_n (t)=\Phi_{n+1}(t)+2\sum_{l=1}^n \theta_{a_l}(t),$ where $\Phi_{n+1}$ is the phase of the Hardy space function $\tilde{f}_{n+1}.$  Combining the corresponding results for inner and outer functions (\cite{QianPhase}) there holds
\[ \int_0^{2\pi} \Phi'_{n+1}(t)dt \ge 0.\]
Therefore, we have
\[ \int_0^{2\pi} \Psi'_n(t)dt \ge 4n\pi.\]
This shows that $\cos\Psi_n(t)$ has at least $4n$ zeros.
The proof is complete.

\section{$n$-Best D-AFD}

The existing $n$-Best AFD (\cite{QQD}) can be accordingly revised.
We first define some auxiliary concepts. If $f\in H^2(\mathbb{D}),$ and there exists
$a_1,\cdots,a_n \in \mathbb{D}$ and $c_1,\cdots,c_n\in \mathbb{C}$ with $c_n\ne 0$ such that

\begin{align}\label{third}
f(z)=\sum^n_{k=1}c_k \tilde{B}_k(z),
\end{align}
then we say that $f$ is of an $n$-\emph{double Blaschke form}. Note that, similar to the AFD case, with the formulation (\ref{second}) and as a consequence of (\ref{must}), we
have an alternative representation of an $n$-double Blaschke form
\begin{align}
f(z)=\sum^n_{k=1}\langle {f},\tilde{B}_k\rangle \tilde{B}_k(z).
\end{align}

\begin{remark}\label{projection}
Denote by $P_{a_1,\cdots,a_k}$ as the projection operator from $H^2(\mathbb{D})$ to ${\rm span}\{\tilde{B}_j\}_{j=1}^k$ and $Q_{a_1,\cdots,a_k}=I-P_{a_1,\cdots,a_k}$ as projection to the orthogonal complement of ${\rm span}\{\tilde{B}_j\}_{j=1}^k,$ then, as a consequence of the relation (\ref{fract}), there holds
\[
\frac{Q_{a_1,\cdots,a_k}}{\phi^2_{a_1}\cdots\phi^2_{a_k}}=
\frac{Q_{a_k}}{\phi^2_{a_k}}\circ\cdots\circ\frac{Q_{a_1}}{\phi^2_{a_1}}.\]
For AFD in the Hardy spaces there exist counterpart results.
We note that the relevant proofs use the particular constructive structure of the TM and the D-TM systems that are available only for the Hardy spaces (of the unit disc and the upper-half complex plane). For all the other RKHSs consisting of holomorphic functions there do not seem to exist similar results.
\end{remark}

The $n$-Best AFD, being essentially equivalent with the best approximation by fractional functions of degrees not exceeding $n,$ has been well studied (\cite{QQD, QQLZ,JHQW, JQSW}). In relation to D-AFD a similar result may be proved. We will adopt the following notation: Denote by  $^{\bf a}\tilde{B}_k=\tilde{B}_k,$ where $\{\tilde{B}_k\}_{k=1}^n$ is the D-TM system generated by the $n$-sequence ${\bf a}=(a_1,\cdots,a_n).$

\begin{theorem}\label{nBest}
  Let $n$ be given and $f$ be in $H^2(\mathbb{D}).$ If $f$ is not an $m$-double-Blaschke form with $m<n,$ then there exists an
  $n$-tuple ${\bf a}=(a_1,\cdots,a_n)$ such that
  \[ \|f-\sum_{j=1}^n\langle f,{}^{\bf a}\tilde{B}_j\rangle {}^{\bf a}\tilde{B}_j\|=\inf_{{\bf b}\in
  {\mathbb D}^n} \{\|f-\sum_{j=1}^n\langle f,{}^{\bf b}\tilde{B}_j\rangle {}^{\bf b}\tilde{B}_j\|\}.\]
\end{theorem}

\begin{remark}
There exist a number of alternative
proofs for existence of the classical $n$-Best AFD of which the earliest studies can refer to
\cite{W,B1} and the references thereby. The latest published proofs include \cite{QQD,QQLZ,JHQW,JQSW}.The reason of the interest and repeated proofs are due to significance  of the problem, theoretical and practical, complication of the old proofs, especially intentions of generalization to Hilbert spaces with a dictionary, including RKHSs, and in particular those consisting of holomorphic functions (\cite{QQLZ}), as well as random versions arising from practical problems \cite{QSAFD,QQD,QZLQ}.  The existing proofs can be divided into two groups:
(i) the order of parameters in ${\bf a}$ is changeable; (ii) the order is unchangeable. The group (i) is for the cases where the used orthonormal system is formulated by Gram-Schmidt orthogonalization. The new case in the present paper belongs to the class (ii): When an $n$-tuple ${\bf a}$ is selected according to the D-AFD selection principle (\ref{optimal}) and ${\bf b}$ is a non-trivial permutation of ${\bf a},$ then, ${\rm span}\{^{\bf a}\tilde{B}_k\}_{k=1}^n\ne{\rm span}\{^{\bf b}\tilde{B}_k\}_{k=1}^n,$ in general. Due to this reason most proofs in the previous studies cannot be adapted to the present case. We, however, can achieve a proof following \cite{JHQW} that is effective to the unchangeable cases. We give a sketch of the proof. For more details the reader is referred to \cite{JHQW}.  \cite{JQSW} deals with a non-commutative hyper-complex case that is also referred to the method of \cite{JHQW}.
\end{remark}

\textit{Proof of Theorem \ref{nBest}.} Basic mathematical analysis shows that through a sequence of $n$-tuples $\{\bf a^{(j)}\}_{j=1}^\infty$ we can reach the infimum value as the corresponding limit. Without loss of generality we can assume that $\{\bf a^{(j)}\}_{j=1}^\infty$ itself is convergent and $\lim_{j\to \infty}{\bf a^{(j)}}={\bf a}=(a_1,\cdots,a_n).$ Assuming that $f$ is not an $m$-double-Blaschke form for any $m<n,$ we show that ${\bf a}\in {\mathbb D}^n,$ and consequently obtain that at ${\bf a}$ the infimum value is attained. We show that  $a_k\in {\mathbb D}, k=1,\cdots,n.$ We first treat the case $k=n.$ We assume the opposite assertion $|a_n|=1$ and will arrive a contradiction. We will show that for the given $f\in H^2({\mathbb D}),$ uniformly for the first $n-1$ parameters,
\begin{align}\label{totoshow} \lim_{j\to \infty}|\langle f,^{{\bf a}^{(j)}}\tilde{B}_n\rangle|^2=0.\end{align}
When this happened it shows that the last parameter sequence  $\{a^{(j)}_n\}_{j=1}^\infty$ just has no role with the approximation, and one could delete the whole sequence without interfering the converge to the above assumed infimum.

To show (\ref{totoshow}), through a density argument, we may assume that $f$ is a bounded function, say $|f|\leq M.$ Assuming $\lim_{j\to\infty}|a^{(j)}_n|=1,$ then the last term in the summation gives rise to the energy
\begin{eqnarray}\label{GBVC} |\langle f,^{{\bf a}^{(j)}}\tilde{B}_n\rangle|^2\leq
M(1-|a^{(j)}_n|)^2\left(\int_0^{2\pi}\frac{1}{|1-\overline{a^{(j)}_n}e^{it}|})\right)^2\to 0,\quad {\rm as}\ j\to \infty,\end{eqnarray}
proved through an elliptic integral (\cite{JHQW}).   Next we show, for any $k_0=1,\cdots,n-1,$ the limit $a_{k_0}$ of $a^{(j)}_{k_0}$ as $j\to \infty,$ cannot be on the unit circle either.
Let $k=k_0$ be any fixed, ranging among $1,\cdots, n-1,$ and $|a_{k_0}|=1.$ Then all the entries in the systems $^{{\bf a}^{(j)}}\tilde{B}_k, k\ge k_0,$
contain, as its factor, either $e_{a^{(j)}_{k_0}}$ or $\phi^{2}_{a^{(j)}_{k_0}}.$ The $j$'s $n$-double-TM system is $\{^{{\bf a}^{(j)}}\tilde{B}_k\}_{k=1}^n,$ and precisely,
\begin{eqnarray*}
^{{\bf a}^{(j)}}\tilde{B}_1,\cdots,^{{\bf a}^{(j)}}\tilde{B}_{k_0},\cdots,
^{{\bf a}^{(j)}}\tilde{B}_n.
\end{eqnarray*}
As $j\to\infty,$ due to the just proved (\ref{GBVC}), and the limiting relation
\[ \lim_{j\to\infty}\phi_{{a_{k_0}}^{(j)}}(e^{it})=-a_{k_0}\quad \ {\rm a.e.\ } t\in [0,2\pi],\]
 proved in (\cite{JHQW}), by invoking the Lebesgue dominated convergence theorem, the terms of the above
 $n$-tuples converge, as $j\to \infty,$  respectively, in the weak-$H^2$-sense to
\begin{eqnarray*}
e_{a_1},\cdots,e_{a_{k_0-1}}\phi^2_{a_1}\cdots\phi^2_{a_{k_0-2}} ,0, e_{a_{k_0+1}}\phi^2_{a_1}\cdots\phi^2_{a_{k_0-1}}(-a_{k_0})^2,\cdots,
e_{a_n}\phi^2_{a_1}\cdots\phi^2_{a_{k_0-1}}(-a_{k_0})^2\phi^2_{a_{k_0+1}}\cdots\phi^2_{a_{n-1}}.
\end{eqnarray*}
It amounts to saying that if the limit $a_{k_0}$ of $a^{(j)}_{k_0}$ belongs to $\partial {\mathbb D},$ then this sequence has no contribution to the optimization. All such columns $a^{(j)}_{k_0}$ have no role and may be deleted.  This implies that $f$ is, in fact, an m-double Blaschke form for some $m<n.$ The proof is complete.

 \section{Superperformance in Reconstruction, Experimental Examples}
 In this section, we present three numerical examples to illustrate and compare the nonlinear approximation performance of the classical core AFD and the proposed D-AFD.
The comparison is conducted from the viewpoint of $n$-term approximation in the $L^2$ sense.

The three test functions are designed to reflect different levels of regularity and structural complexity.
The first example is a piecewise smooth periodic function with derivative discontinuity, representing a typical scenario in which the target function exhibits structural transitions while remaining global regularity.
The second example consists of a localized spike superimposed on a smooth oscillatory background, mimicking functions with concentrated transient components and reduced local regularity.
The third example is constructed as a linear combination of squared Blaschke products, which serves as a model problem whose intrinsic structure is closely related to the modified double Blaschke forms employed by the D-AFD.

For all three examples, we report the relative $L^2$ reconstruction errors with respect to the number of approximation terms, and visual comparison of the reconstructed signals obtained by the two methods.
All experiments are performed under the same parameter settings and stopping criteria, in order to ensure a fair and objective comparison of their approximation capabilities.

\begin{example} We consider the following piecewise function on $[0,2\pi]$:
\begin{equation}
\label{eq:example1_balanced_final}
f(t) =
\begin{cases}
\displaystyle \sin(4t) - \frac{1}{4} t,
& t \in [0,\pi],\\[8pt]
\displaystyle \sin(4t) + \frac{1}{2}(t-\pi) - \frac{1}{4} t,
& t \in (\pi,2\pi].
\end{cases}
\end{equation}
The function is continuous on $[0,2\pi]$ while its first derivative has a jump discontinuity at $t=\pi$.
This example is used to evaluate the approximation performance of the core AFD and the D-AFD for functions with reduced regularity caused by structural transitions.
\end{example}
The reconstruction results of Example~1 obtained by the core AFD and the D-AFD are shown in Fig.~\ref{f1}. The dot and un-dot black curves, blue curves and the red curves represent, respectively, the graph of the given function $f,$  approximations by core AFD and by D-AFD with different orders (terms).
As the number of terms increases, both methods provide progressively improved approximations of the target function. In particular, in Picture 1, the $6$-term D-AFD and the $12$-term core AFD are nearly identical, and both are very close to the true graph of $f.$ Noticeably, the $6$-term AFD has remarkably more errors. Picture 2 again repeats this phenomenon: $16$-core AFD is almost the same as $8$-D-AFD, while $8$-core AFD is far away. Picture 3 shows that $10$-D-AFD and $20$-core AFD visually coincide with the true graph while $10$-core AFD is far away. In summary, with high efficiency D-AFD reconstructs the given signal with approximately double speed as that of core AFD. This is further reflected by the smaller residual amplitudes shown in Fig.~\ref{f2}.
This advantage is quantitatively confirmed by the relative $L^2$ error decay curves in Fig.~\ref{f3}, where the D-AFD exhibits a substantially faster error reduction as $n$ increases.
For the two types of methods the distributions of the selected parameters $\{a_k\}_{k=1}^{10}$ are shown in Fig.~\ref{f4}: We observe that for D-AFD the parameters $\{a_k\}$ are more concentrate in the center of the disc. This would reduce the singularity, and thus the computation as well, and increase the accuracy.

 	\begin{figure}[H]
       \scriptsize
		\centering
		\subfigure{
			\includegraphics[width=4.5cm]{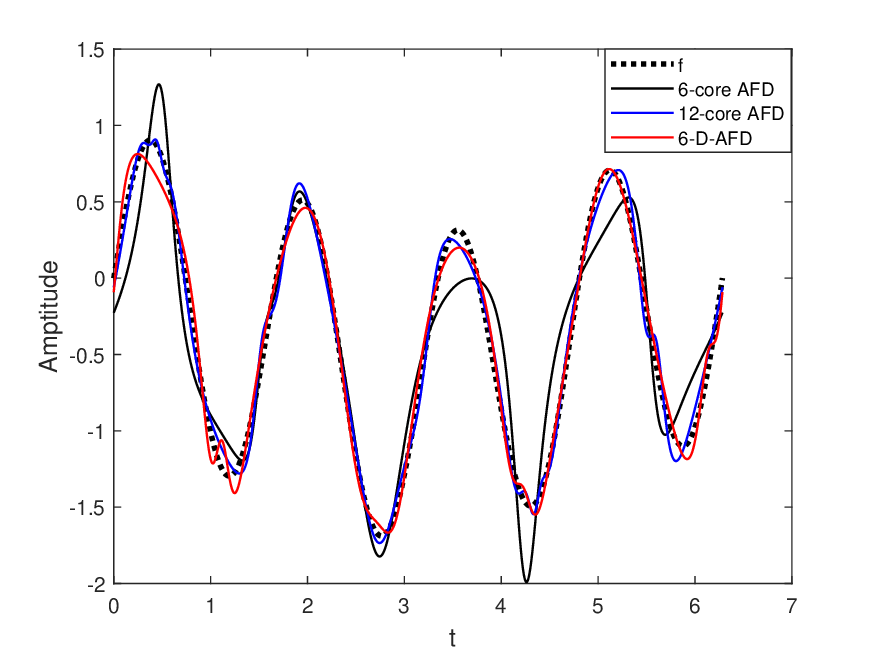}
		}
		\quad
		\subfigure{
			\includegraphics[width=4.5cm]{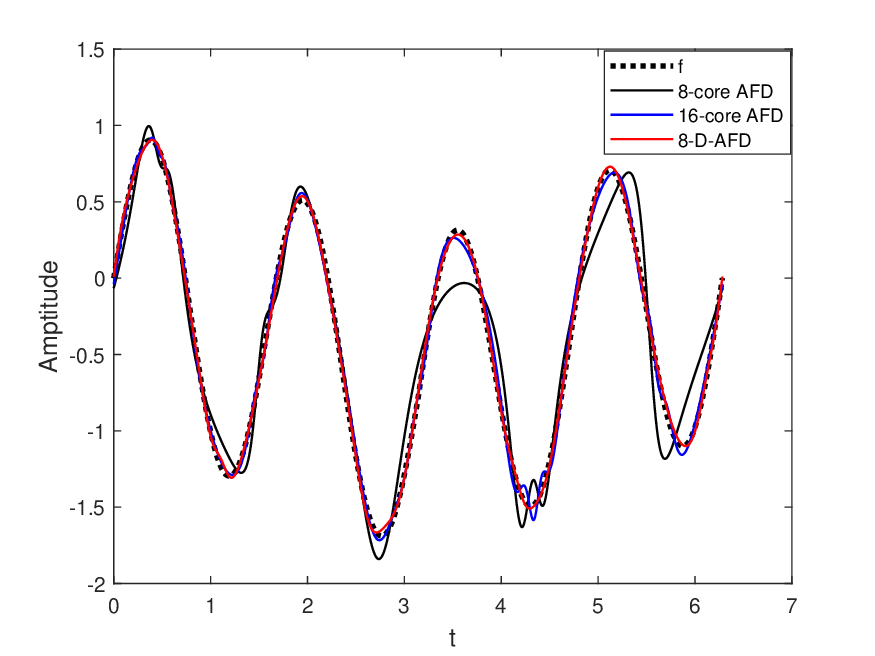}
		}
		\quad
		\subfigure{
			\includegraphics[width=4.5cm]{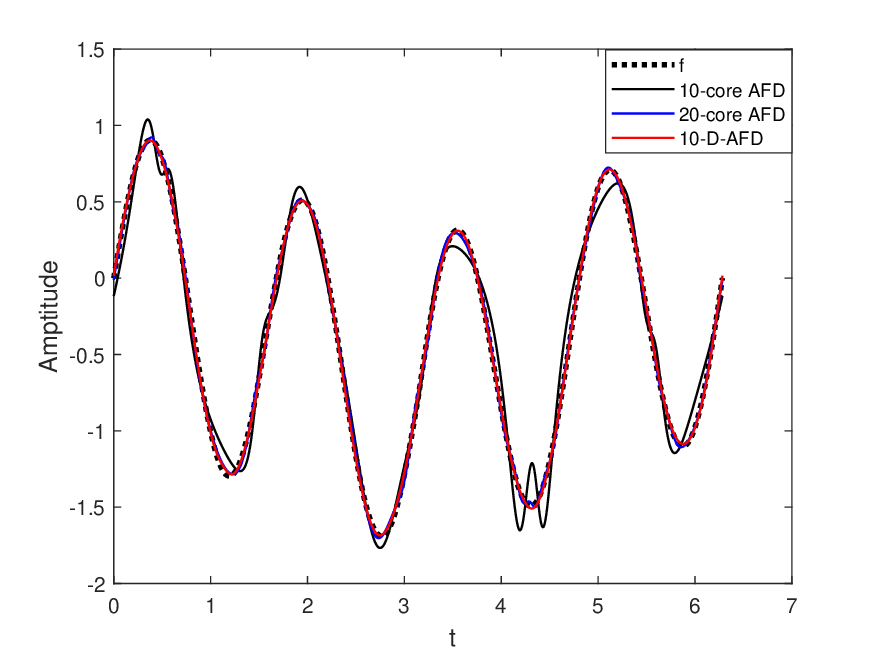}
		}
		\caption{Reconstruction results of Example 1 obtained by core AFD (blue) and D-AFD (red), compared with the target function $f$ (black dashed). }
		\label{f1}		
	\end{figure}
 	\begin{figure}[H]
       \scriptsize
		\centering
		\subfigure{
			\includegraphics[width=4.5cm]{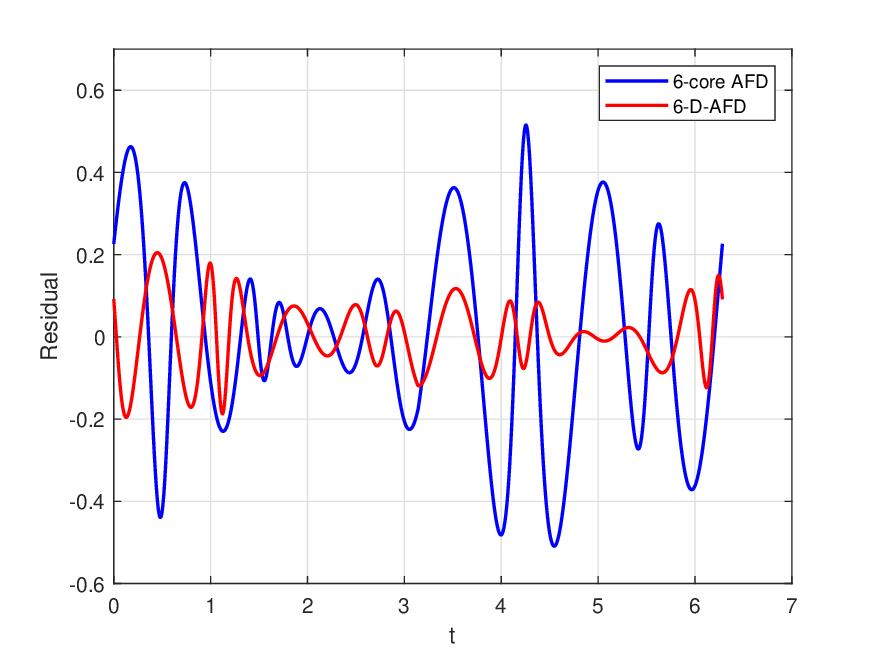}
		}
		\quad
		\subfigure{
			\includegraphics[width=4.5cm]{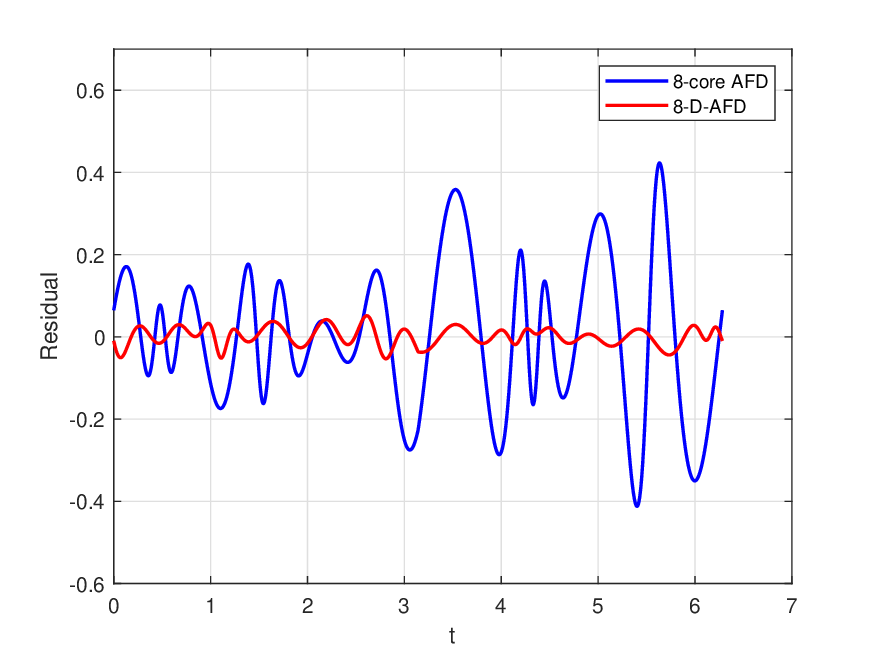}
		}
		\quad
		\subfigure{
			\includegraphics[width=4.5cm]{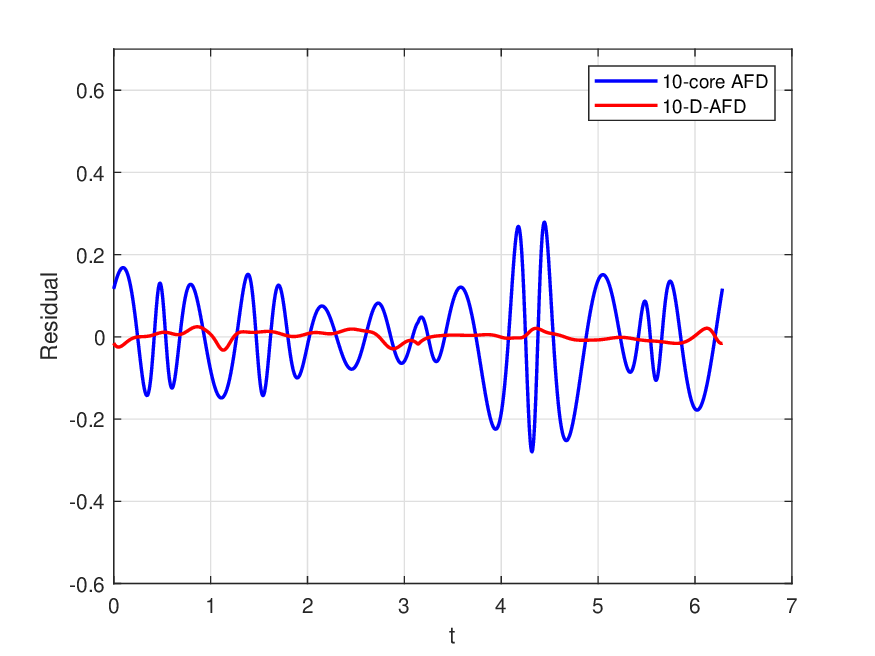}
		}
		\caption{Residuals of the $n$-term approximations for Example~1 obtained by core AFD and D-AFD.}
		\label{f2}		
	\end{figure}
\begin{figure}[H]
\centering
\begin{minipage}[t]{0.48\textwidth}
    \centering
    \includegraphics[width=\linewidth]{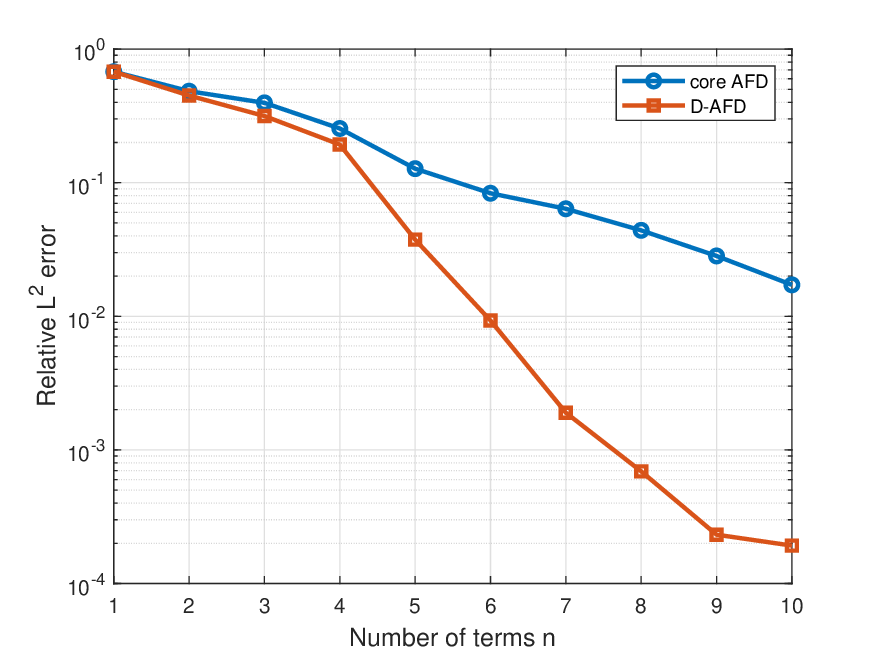}
    \caption{Relative $L^2$ error decay versus the number of terms $n$.}
    \label{f3}
\end{minipage}\hfill
\begin{minipage}[t]{0.48\textwidth}
    \centering
    \includegraphics[width=\linewidth]{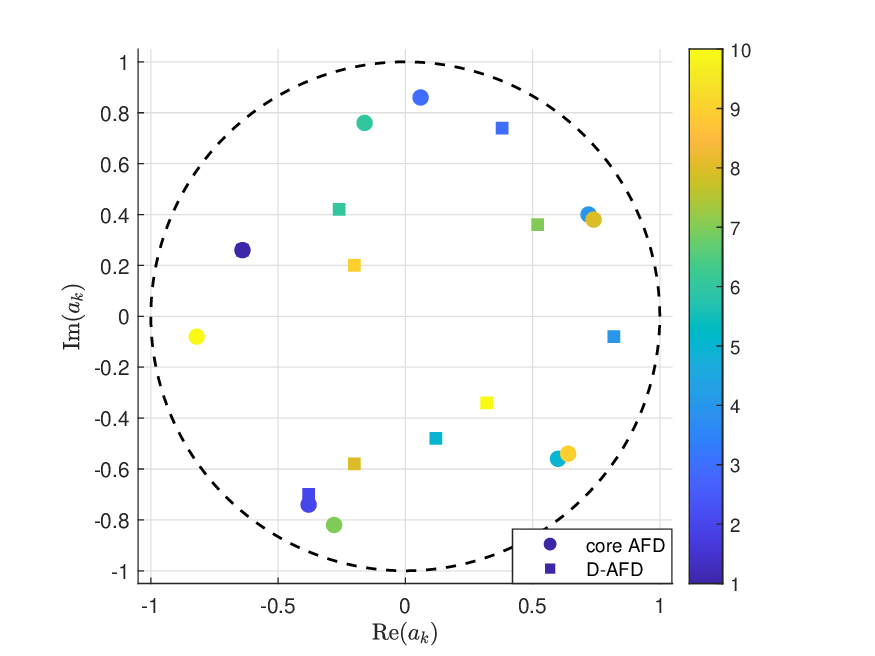}
    \caption{Distribution of the selected parameters $a_k$($k=1,\ldots,10$) in the unit disk.}
    \label{f4}
\end{minipage}
\end{figure}

\begin{example}\label{ex:szego_linear_comb}
We consider the following function on $[0,2\pi]$ defined by a linear combination of five Szeg\"{o} kernels:
\begin{equation}\label{eq:ex2_szego}
f(t)=\sum_{k=1}^{5} c_k\,K_{a_k}\!\left(e^{it}\right),\qquad t\in[0,2\pi],
\end{equation}
where $K_a$ denotes the Szeg\"{o} kernel on the unit disk $\mathbb{D}$,
\begin{equation}\label{eq:szego_kernel}
K_a(z)=\frac{1}{1-\overline{a}\,z},\qquad z\in\mathbb{D},\ |a|<1,
\end{equation}
and the parameters $\{a_k\}_{k=1}^5\subset\mathbb{D}$ and coefficients $\{c_k\}_{k=1}^5\subset\mathbb{C}$ are chosen to be
\begin{equation}\label{eq:ex2_params}
\begin{aligned}
(a_1,a_2,a_3,a_4,a_5)
&=\bigl(0.20+0.20i,\ 0.55-0.15i,\ -0.30+0.40i,\ 0.75+0.05i,\ -0.10-0.60i\bigr),\\
(c_1,c_2,c_3,c_4,c_5)
&=\bigl(1.0,\ -0.7,\ 0.4,\ 0.9,\ -0.5\bigr).
\end{aligned}
\end{equation}

This example is used to evaluate the approximation performance of the core AFD and the D-AFD on a rational-type function generated by Szeg\"{o} kernels, whose analytic structure is compatible with Hardy space modeling on the unit disk.
\end{example}

The numerical reconstructions of Example~2 produced by the core AFD and the D-AFD are displayed in Fig.~\ref{f11} for several truncation levels. As the number of terms increases, both approximations capture the overall profile of the target function increasingly well. For each fixed $n$, however, the D-AFD approximation follows the target curve more faithfully, especially around the regions with rapid variation, which is also evidenced by the reduced oscillation magnitude in the corresponding residual plots shown in Fig.~\ref{f12}. A quantitative comparison is provided in Fig.~\ref{f13}, where the relative $L^{2}$ error of the D-AFD decreases more rapidly with respect to $n$, indicating a more efficient use of terms. Finally, Fig.~\ref{f14} reports the locations of the selected parameters $a_k$ in the unit disk, suggesting that the two methods explore distinct geometric sampling patterns, which correlates with their different convergence behaviors.

 	\begin{figure}[H]
       \scriptsize
		\centering
		\subfigure{
			\includegraphics[width=4.5cm]{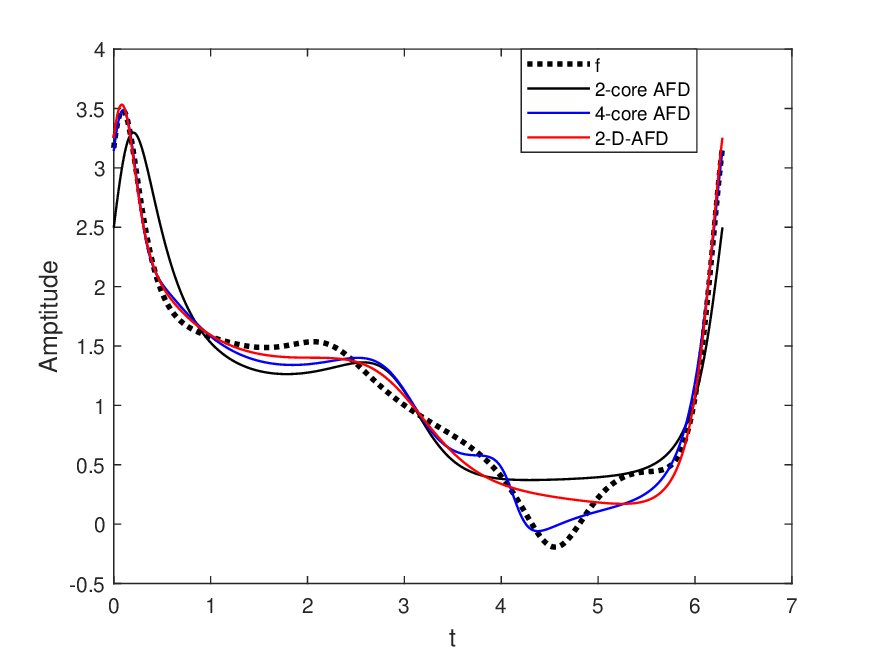}
		}
		\quad
		\subfigure{
			\includegraphics[width=4.5cm]{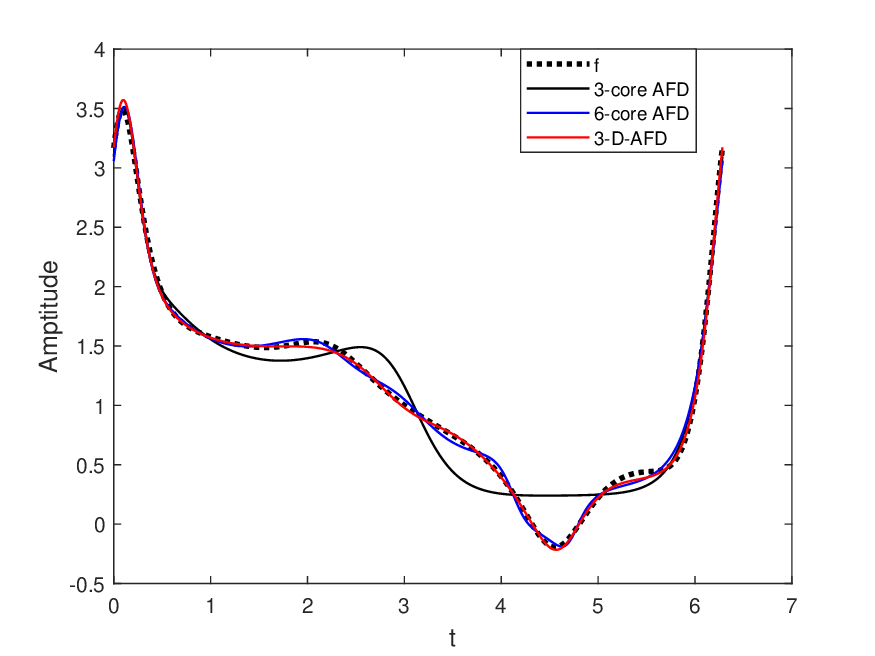}
		}
		\quad
		\subfigure{
			\includegraphics[width=4.5cm]{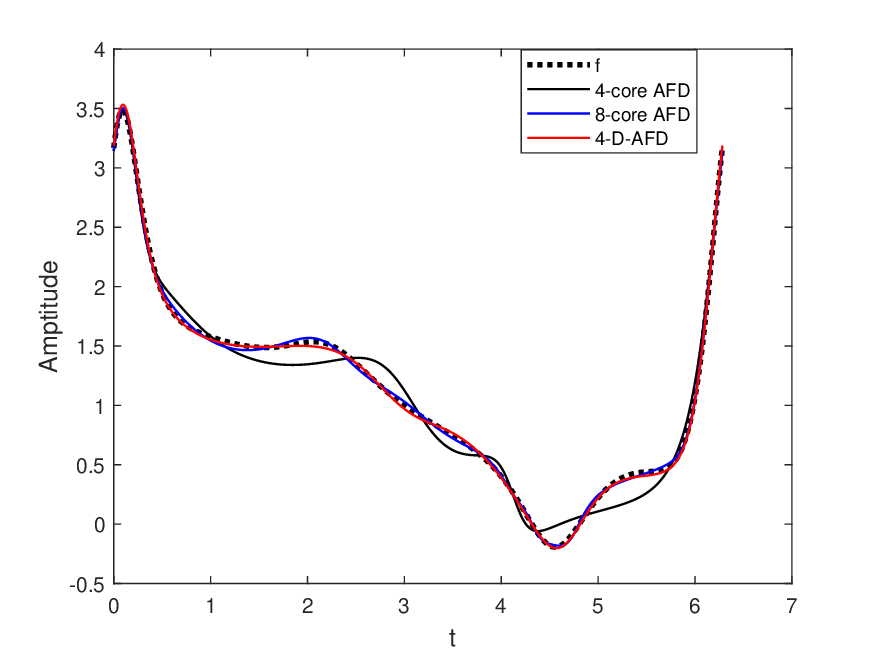}
		}
		\caption{Reconstruction results of Example 2 obtained by core AFD (blue) and D-AFD (red), compared with the target function $f$ (black dashed).}
		\label{f11}		
	\end{figure}
 	\begin{figure}[H]
       \scriptsize
		\centering
		\subfigure{
			\includegraphics[width=4.5cm]{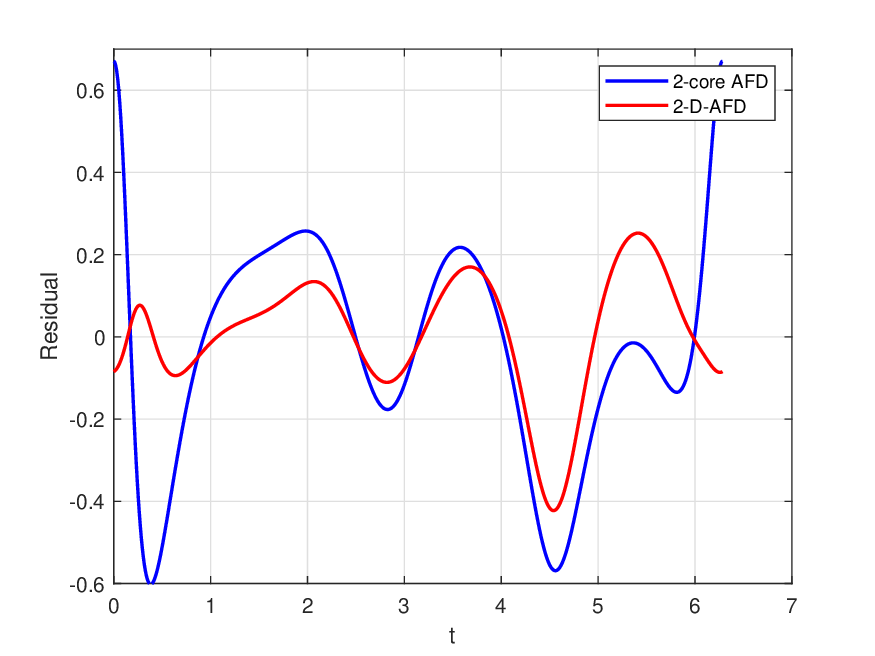}
		}
		\quad
		\subfigure{
			\includegraphics[width=4.5cm]{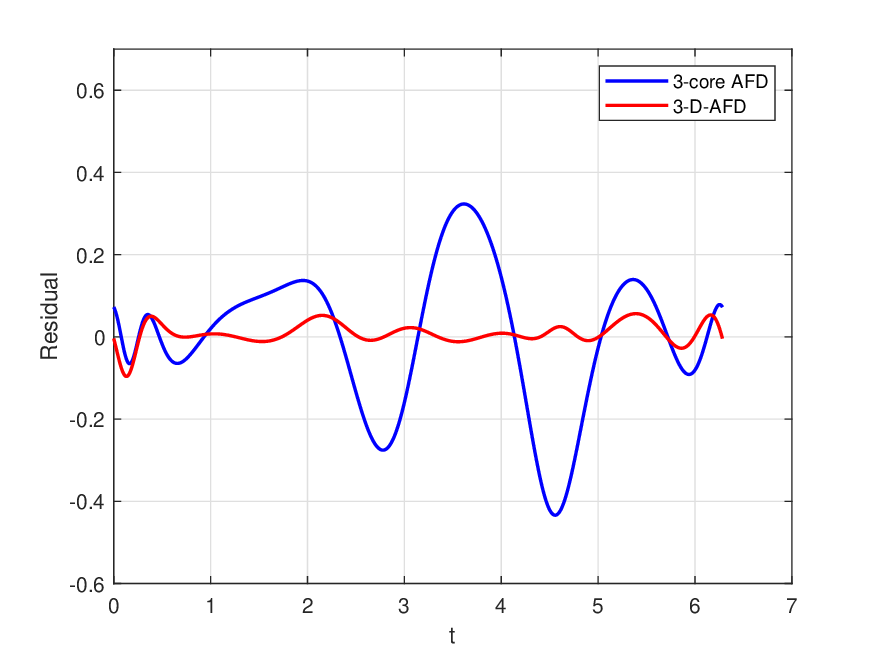}
		}
		\quad
		\subfigure{
			\includegraphics[width=4.5cm]{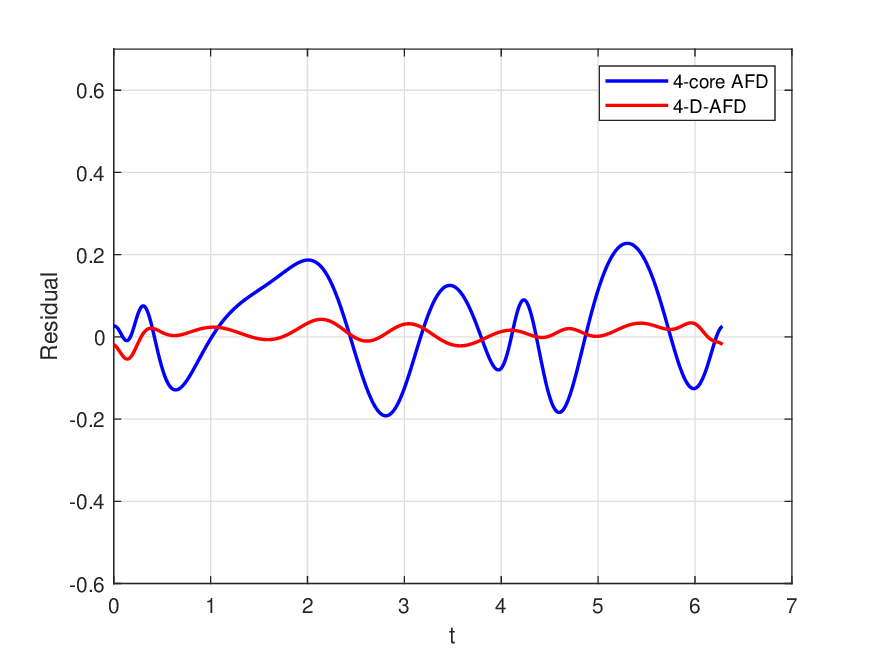}
		}
		\caption{Residuals of the $n$-term approximations for Example~2 obtained by core AFD and D-AFD.}
		\label{f12}		
	\end{figure}

\begin{figure}[H]
\centering
\begin{minipage}[t]{0.48\textwidth}
    \centering
    \includegraphics[width=\linewidth]{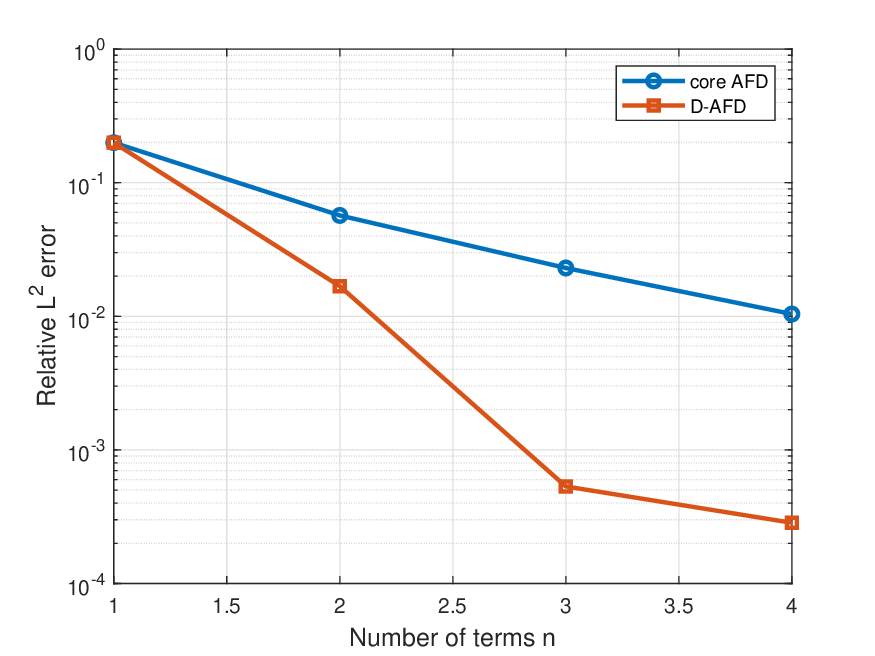}
    \caption{Relative $L^2$ error decay versus the number of terms $n$.}
    \label{f13}
\end{minipage}\hfill
\begin{minipage}[t]{0.48\textwidth}
    \centering
    \includegraphics[width=\linewidth]{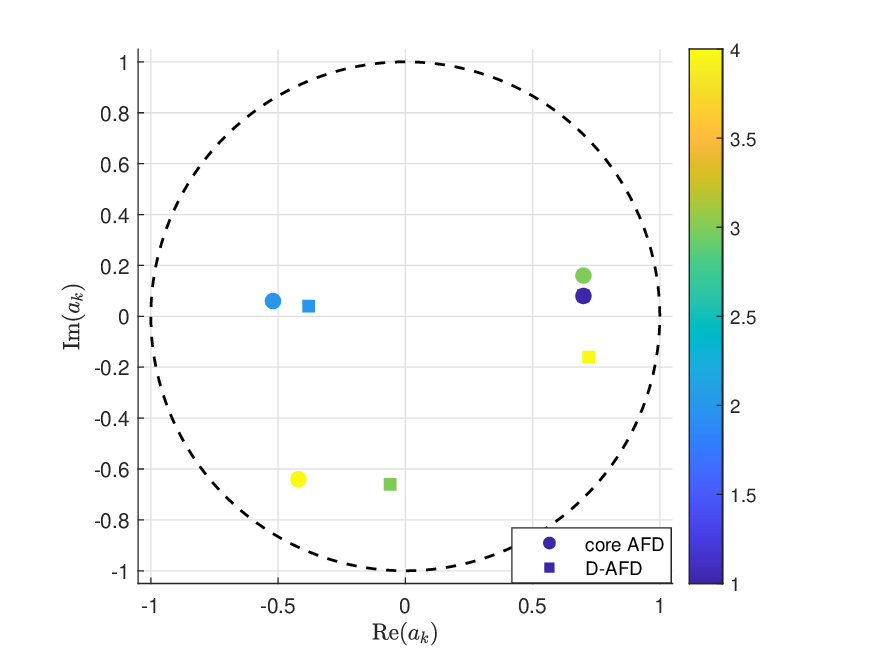}
    \caption{Distribution of the selected parameters $a_k$($k=1,\ldots,4$) in the unit disk.}
    \label{f14}
\end{minipage}
\end{figure}

\section{Higher Order Generalization and Concluding Remarks and Questions}

Let $f$ belong to $H^2(\mathbb{D}).$
Denote, for $a\in \mathbb{D}, f_2(z)=\frac{f(z)-\langle f,e_a\rangle e_a(z)}{\frac{z-a}{1-\overline{a}z}}=-\overline{a}f(z)+(1-|a|^2)\frac{f(z)-f(a)}{z-a}, z\ne a.$ In this section
 $a$ is not necessary to be optimally selected according to the energy principle. The following is a generalization of Lemma \ref{ccore}.

\begin{lemma}\label{finite sum generalization}
  Let $f_2\in H^2(\mathbb{D})$ be as above defined. If $a$ satisfies the differential equation
  \begin{eqnarray}\label{re}-\overline{a}f^{(k-1)}(a)+\frac{1-|a|^2}{k}f^{(k)}(a)=0,\end{eqnarray}
  then
  \[ \lim_{z\to a} \frac{f_2(z)}{(z-a)^{k-1}}=0.\]
 For such $a$, there exists $\tilde{f}_2(z)$ belonging to $ H^2(\mathbb{D})$ such that
 $$\tilde{f}_2(z)=\frac{\tilde{g}_2(z)}{\left(\frac{z-a}{1-\overline{a}z}\right)^{k+1}},$$
 where $\tilde{g}_2=f-\langle f,e_a\rangle e_a.$
\end{lemma}

Note that when $k=1$ the relation (\ref{re}) is a consequence of the optimality of $a_1$ proved in Lemma \ref{ccore}.

\textit{Proof of Lemma \ref{finite sum generalization}.}
Using the Taylor expansion for $f$ in the expression of $f_2,$ followed by applying $k-1$ times of the $L'H\hat{o}pital$ principle, there hold
 \begin{eqnarray*} \lim_{z\to a}\frac{f_2(z)}{(z-a)^{k-1}}&=&
 \frac{1}{(k-1)!}\lim_{z\to a}\left[-\overline{a}f^{(k-1)}(z)+\frac{1-|a|^2}{k}f^{(k)}(z)\right]\\
 &=&\frac{1}{(k-1)!}\left[-\overline{a}f^{(k-1)}(a)+\frac{1-|a|^2}{k}f^{(k)}(a)\right].\end{eqnarray*}
 When the last quantity is equal to zero, it concludes that the quotient $\frac{f_2(z)}{(z-a)^{k-1}}$ is a function in $H^2(\mathbb{D}).$ The proof is complete.

 \begin{theorem}\label{ge}
   Let $k$ be arbitrary but fixed positive integer and $a_l,l=1,\cdots,n,$ are inside $\mathbb{D}$ satisfying, respectively, the equations
   \begin{eqnarray}\label{eq}-\overline{a_l}\tilde{f}_l^{(k-1)}(a_l)+\frac{1-|a_l|^2}{k}\tilde{f}_l^{(k)}(a_l)=0,\end{eqnarray}
 where $$\tilde{f}_l(z)=\frac{\tilde{g}_l(z)}{\left(\frac{z-a}{1-\overline{a}z}\right)^{k+1}}$$
 and
 $$\tilde{g}_l(z)=f(z)-\sum_{j=1}^{l-1}\langle \tilde{f}_j,e_{a_j}\rangle e_{a_j}(z)\prod_{s=1}^{j-1}
\left(\frac{z-a_s}{1-\overline{a}_sz}\right)^{k+1}
 $$
 are well defined and satisfy
 $$f(z)=\sum_{j=1}^{n}\langle \tilde{f}_j,e_{a_j}\rangle e_{a_j}(z)\prod_{s=1}^{j-1}
\left(\frac{z-a_s}{1-\overline{a}_sz}\right)^{k+1}+\tilde{f}_{n+1}(z)\prod_{s=1}^{n}
\left(\frac{z-a_s}{1-\overline{a}_sz}\right)^{k+1},$$ and $\tilde{f}_{n+1}\in H^2(\mathbb{D}).$
 \end{theorem}

The condition (\ref{eq}) for the case $k=1$ of Theorem \ref{ge} is a consequence of the optimality of $a_1,\cdots,a_n$ as proved in Lemma \ref{ccore}. The optimality also guarantees the convergence and the convergence rate results proved in \S 3 and \S 4, respectively. For $k>1,$ the theorem provides a sufficient condition  (\ref{eq}) for holdness of a $(k+1)$-order interpolation, without concerning the convergence aspect as $n\to \infty.$ The condition (\ref{eq}) itself, on the other hand, is hard to be verified. Therefore, a triple or any other higher order generalization beyond double interpolation (\ref{second}) and Double AFD (\ref{series}) do not seem to be easily available.

\begin{remark}
  Experiments show that in the reconstruction aspect Double AFD superperforms Core AFD. Experiments and knowledge in digital signal processing all point that Double AFD is an improvement of Core AFD. However, one would need a rigorous proof. On the other hand, theoretical proofs or comparison of effectiveness of greedy type algorithms are difficult, for a present un-optimal selection may end up with better results in the future.
\end{remark}

\begin{remark}
  It is known that TM systems are Gram-Schmidt orthogonalization of the Szeg\"o (reproducing) kernels (Appendix of \cite{QSAFD}). It is a question whether a Double TM system is orthogonalization of a sequence basic functions in a dictionary.
\end{remark}

\begin{remark}
It does not seem that D-AFD can be easily extended to random fields $f(w,t)$ to form, say, a stochastic D-AFD \cite{QSAFD}. If could, it would  mean that there exists a single selection of a parameter $a$ that, for a.s. $w,$ gives rise to the remainders
$f(w,z)-\langle f_w,e_a\rangle e_a(z)$ all having the common factor $(z-a)^2.$ A common factor $(z-a)$ may be achieved due to the reproducing kernel property. A double zero factor, on the other hand, as seen, depends on optimality of the same $a$ for a.s. $w.$ The same reasoning causes the defeat of generalization to multivariate cases (functions of several complex variables). It is the same reason that makes a stochastic generalization of Blaschke unwinding expansion unavailable (\cite{CS,CP,Qinner}).
\end{remark}

\begin{remark}
All the concepts and techniques of double interpolation and D-AFD may be adapted to the upper-half complex plane context. The theme may be extended to contexts beyond Hardy spaces. We will discuss some general cases in a separate paper.
\end{remark}

 \vspace{0.1in}
 {\bf{Acknowledgement.}}

This work was supported by the Major Project of Guangzhou National Laboratory [grant number GZNL2024A01004]; the Science and Technology Development Fund of Macau SAR [grant number 0015/2025/AFJ]; the National Natural Science Foundation of China [grant number 12401127];and the WuHan Textile University[grant number 20230612].

 \vspace{0.1in}
 {\bf{Conflicts of Interest.}}
The authors declare that there are no conflicts of interest.


\begin{thebibliography}{99}

\bibitem{B1}
L. Baratchart, M. Olivi, F. Wielonsky,
On a rational approximation problem in the real Hardy space $H_2$,
Theor. Comput. Sci. 94 (2) (1992) 175--197.


\bibitem{CS}
R. Coifman, S. Steinerberger,
Nonlinear phase unwinding of functions,
J. Fourier Anal. Appl. 23 (2017) 778--809.

\bibitem{CP}
R.R. Coifman, J. Peyri\`ere,
Phase unwinding, or invariant subspace decompositions of Hardy spaces,
J. Fourier Anal. Appl. 25 (2019) 684--695.

\bibitem{DT}
R.A. DeVore, V.N. Temlyakov,
Some remarks on greedy algorithms,
Adv. Comput. Math. 5 (1996) 173--187.

\bibitem{Gar}
J.B. Garnett,
\emph{Bounded Analytic Functions},
Academic Press, New York, 1987.

\bibitem{JHQW}
M. Jin, Z.-T. Han, T. Qian, J.-X. Wang,
A direct proof for existence of best approximation of random signals,
Complex Anal. Oper. Theory 20 (2026) 19-20.

\bibitem{JQSW}
M. Jin, T. Qian, I. Sabadini, J.-X. Wang,
N-best adaptive Fourier decomposition for slice hyperholomorphic functions,
Adv. Math. 480 (2025) 110498.

\bibitem{LQ}
C. Lin, T. Qian,
Frequency analysis with multiple kernels and complete dictionary,
Appl. Math. Comput. 466 (2024) 128477.

\bibitem{QianPhase}
T. Qian,
Boundary derivatives of the phases of inner and outer functions and applications,
Math. Methods Appl. Sci. 32 (2009) 253--263.

\bibitem{Qinner}
T. Qian,
Intrinsic mono-component decomposition of functions: An advance of Fourier theory,
Math. Methods Appl. Sci. 33 (7) (2010) 880--891.

\bibitem{QSAFD}
T. Qian,
A sparse representation of random signals,
Math. Methods Appl. Sci. 45 (8) (2022) 4210--4230.

\bibitem{QWa1}
T. Qian, Y.B. Wang,
Adaptive Fourier series---a variation of greedy algorithm,
Adv. Comput. Math. 34 (2010) 279--293.

\bibitem{QWa2}
T. Qian, Y.B. Wang,
Remarks on adaptive Fourier decomposition,
Int. J. Wavelets Multiresolut. Inf. Process. 11 (1) (2013) 1350008.


\bibitem{QQD}
W. Qu, T. Qian, G.-T. Deng,
A stochastic sparse representation: N-best approximation to random signals and computation,
Appl. Comput. Harmon. Anal. 55 (2021) 185--198.

\bibitem{QQLZ}
W. Qu, T. Qian, H.C. Li, K.H. Zhu,
Best kernel approximation in Bergman spaces,
Appl. Math. Comput. 416 (2022) 126749.

\bibitem{QZLQ}
T. Qian, Y. Zhang, W. Liu, W. Qu,
Adaptive Fourier decomposition-type sparse representations versus the Karhunen--Lo\`eve expansion for decomposing stochastic processes,
Math. Methods Appl. Sci. 46 (13) (2023) 14007--14025.

\bibitem{W}
J.L. Walsh,
\emph{Interpolation and Approximation by Rational Functions in the Complex Plane},
American Mathematical Society, Providence, RI, 1969.

\end{thebibliography}
\end{document}